\documentclass[preprint,12pt]{elsarticle}
\usepackage{amssymb}
\usepackage{amsmath}
\usepackage{amsthm}

\usepackage{mathrsfs}
\usepackage{amsfonts}
\usepackage{color}
\usepackage{pgfplots}
\usepackage[colorlinks=true,linkcolor=blue, urlcolor=red, citecolor=blue]{hyperref}

\usepackage{multirow}
\usepackage{float}

\newtheorem{theorem}{Theorem}[section]
\newtheorem{lemma}[theorem]{Lemma}
\newtheorem{proposition}[theorem]{Proposition}
\newtheorem{corollary}[theorem]{Corollary}
\newtheorem{maintheorem}{Theorem}

\theoremstyle{definition}
\newtheorem{definition}[theorem]{Definition}
\newtheorem{example}[theorem]{Example}

\theoremstyle{remark}
\newtheorem{remark}[theorem]{Remark}

\numberwithin{equation}{section}

\newcommand{\R}{\mathbb{R}}
\newcommand{\N}{\mathbb{N}}
\newcommand{\eps}{\varepsilon}

\newcommand{\wum}{W^{1,p}(\Omega,\mathbb{R}^d)}

\newcommand{\mlp}[1]{\mathcal{M}^{1,p}_\lambda(#1)}
\newcommand{\mpi}{\mathcal{M}^{1,p}(I)}

\newcommand{\dfx}[2]{\frac{\partial #1}{\partial{#2}}}

\journal{Bulletin des Sciences Mathématiques}

\begin{document}
	
	\begin{frontmatter}

		\title{{\color{magenta} Limits of sequences of volume preserving homeomorphisms in $W^{1,p}$, for $0<p<1$}}
		\author[cmat]{Assis Azevedo\corref{cor1}}
		\ead{assis@math.uminho.pt}
		\author[cmat]{Davide Azevedo}
		\ead{davidemsa@math.uminho.pt}

		\cortext[cor1]{Corresponding author}
	
		\address[cmat]{Centro de Matem\'{a}tica, Universidade do Minho,\\
			Campus de Gualtar,
			4700-057 Braga, Portugal}		
		
		\begin{abstract}
		If $\Omega$ is an open subset of $\R^d$ and $p>0$ then the elements of $W^{1,p}(\Omega)$ can be seen as the pairs $(f,F)\in L^p(\Omega)\times (L^p(\Omega))^d$ such that there exists a sequence $(f_n)_n$ of $C^1$ functions converging to $f$ in $L^p(\Omega)$ such that $(\nabla f_n)_n$ converges to $F$ in $(L^p(\Omega))^d$.  If $p\geq 1$ the pair $(f,F)$ is defined by $f$ as $F$ must be the distributional gradient of $f$. 
		
		If $0<p<1$, there is, in general, a disconnection between $f$ and $F$. For instance,
		Peetre (see \cite{peetre}) proved that, if $d=1$, this disconnection is complete, as any pair $(f,F)\in L^p(\Omega)\times L^p(\Omega)$ is an element of $W^{1,p}(\Omega)$. So $F$ is not defined by $f$ in any sense, as it can be any element of $L^p(\Omega)$.
			
			In this paper we obtain results of this type, concerning $C^1$ homeomorphisms of $\Omega$ that are volume preserving if $d\geq 2$. 
			
			We will show, in particular, that if $H:\Omega\rightarrow SO(d)$ is a Riemann integrable function, then there exists a sequence $(f_n)_n$ of  orientation and volume preserving $C^\infty$ homeomorphisms of $\Omega$ uniformly converging to the identity of $\Omega$ and such that $\left(Df_n\right)_n$ converges to $H$ in $L^p(\Omega)^{d^2}$. 
			
		 If $d=1$ and $I$ is a bounded interval,
		  we will prove that a pair  $(f,F)\in C^1(I)\times L^p(I)$, such that $f'\neq 0$, $f', (f^{-1})'\in L^r(I)$ for some $r>1$, admits a sequence $(f_n)_n$ of $C^1$ homeomorphisms uniformly converging to $f$ and such that $(f_n')_n$ converges in $L^p(I)$ to $F$, if and only if $0\leq \frac{F}{f'}\leq 1$.
		\end{abstract}

		
	\end{frontmatter}
	
	\section{Introduction} 	

In this paper, we are interested in the space $W^{1,p}(\Omega,\R^d)$, especially for $0<p<1$, $d\in\N$ and $\Omega$ a bounded open subspace of $\R^d$. 

Recall that $W^{1,p}(\Omega,\R^d)$, for $0<p<\infty$, is the completion of
\begin{equation*}
	X=\left\{f=(f_1,\ldots,f_d)\in C^1(\Omega)^d: \tfrac{\partial f_i}{\partial x_j}\in L^p(\Omega), \text{ for } i,j=1,\ldots,d\right\}
\end{equation*}
relatively to the norm or quasi-norm, for  $1\leq p<\infty$ or $0<p<1$, respectively, defined by
\begin{equation}\label{norma1p}
	\|f\|_{1,p}=\|f\|_p+\|Df\|_{p},
\end{equation}
where $Df$ is the Jacobian matrix of $f$, $\|f\|_p=\max_i\|f_i\|_p$ and  $\|Df\|_p=\max_{i,j}\left\|\dfx{f_i}{x_j}\right\|_p$.

In fact, if $0<p<1$, the proof that \eqref{norma1p} defines a quasinorm is a consequence of the  inequalities $(a+b)^p\leq a^p+b^p$ and $(a+b)^c\leq 2^{c-1}\left(a^c+b^c\right)$, for $a,b\geq 0$ and $c\geq 1$, from where we obtain
\begin{equation}\label{desigualdadep}
	\forall f,g\in  L^p(\Omega) \quad\left\{
	\begin{array}{ll}
		\|f+g\|^p_p&\leq \|f\|^p_p+\|g\|^p_p\\[2mm]
		\|f+g\|_p & \leq 2^\frac{1-p}{p}\left(\|f\|_p+\|g\|_p\right).
	\end{array}\right.
\end{equation}

We have a canonical continuous inclusion 
\begin{equation*}
	\pi=(\pi_1,\pi_2):\wum\longrightarrow L^p(\Omega)^d\times \left(L^p(\Omega)\right)^{d^2}
\end{equation*} 
defined by: 
if $\phi\in\wum$ is represented by a sequence $(f_n)_n$, then $\pi_1(\phi)$ is the limit of $(f_n)_n$ in  $L^p(\Omega)^d$ and $\pi_2(\phi)$ is the  limit of $(Df_n)_n$ in $\left(L^p(\Omega)\right)^{d^2}$.

Then  $\phi$ is represented by the pair $(\pi_1(\phi),\pi_2(\phi))$ and
\begin{equation*}
\|\phi\|_{1,p}=\|\pi_1(\phi)\|_p+\|\pi_2(\phi)\|_p.
\end{equation*}

 Of course, for $0<p<1$, the inequalities in \eqref{desigualdadep} are still valid in $\wum$.

By \cite{H=W}, if $p\geq 1$, then $\pi_2(\phi)$ is the distributional derivative of $\pi_1(\phi)$ and then $\phi$ will be represented only by $\pi_1(\phi)$, that is, $\pi_1$ is injective. So
\begin{equation*}
	\wum=\left\{f\in  L^p(\Omega)^d: \tfrac{\partial f_i}{\partial x_j}\in L^p(\Omega), \text{ for } i,j=1,\ldots,d\right\},
\end{equation*}
where the derivatives are distributional ones.

If $0<p<1$ the situation is completely different. For instance, for $d=1$, Peetre (see \cite{peetre}) proved that not only $\pi_1$ is not injective but $\pi$ is a bijection.\vskip2mm

The following result will be used in the next section.

\begin{remark}
If $0<p<q\leq \infty$ and $\Omega$ is bounded, then, with the convention that $\frac{\infty-p}{p\infty}=\frac{1}{p}$, we have
	\begin{equation}\label{pqmenor}
		\forall f\in L^q(\Omega)\quad		\|f\|_p\leq
		\lambda(\Omega)^{\frac{q-p}{pq}}\|f\|_q,
	\end{equation}
	from where we obtain that $W^{1,q}(\Omega,\mathbb{R}^d)$ is continuously included in $\wum$.
	
	The proof is the same as in the case where $p\geq 1$. In fact, \eqref{pqmenor}, is equivalent to  $\||f|^p\|_1\leq 	\||f|^p\|_{\frac{q}{p}}\,	\|1\|_{\left(\frac{q}{p}\right)'}$.
\end{remark}
\subsection{Main results}
We now consider a subspace of $W^{1,p}(\Omega,\R^d)$, with $0<p<1$, with a more demanding topology.  More specifically, $\mlp{\Omega}$,  the set formed by the pairs $(f, H)\in L^p(\Omega)\times L^p(\Omega)^{d^2}$
admitting a sequence $(f_n)_n$ of  volume preserving $C^1$ homeomorphisms  of $\Omega$ continuous to the boundary, converging uniformly to $f$ and such that $(Df_n)_n$ converge to $H$ in $L^p(\Omega)^{d^2}$. In these conditions, $f$ also preserves the volume, as it is the uniform limit of volume preserving homeomorphisms.

In the case $d=1$, as there are no non-trivial such homeomorphisms, we will work in $\mathcal{M}^{1,p}(I)$, for $I$ an open bounded interval of $\R$, which is defined as the space above, without the volume preserving condition.

The paper \cite{aabt} makes use of these spaces in the context of Dynamical Systems in Sobolev Spaces, proving that, for a closed connected $d$-dimensional manifold $X$, $d \geq 2$ and $0< p<1$, the set of the ergodic elements in  $\mlp{X}$ is generic, and presenting an example of $(g,h)\in\mlp{I^d}$, for $0<p<1$, such that $h$ is not the distributional derivative of $g$.

Haakan Hedenmalm and Aron Wennman, in \cite{haakan}, explore a sharp phase transition phenomenon which occurs for $L^p$-Carleman classes with exponents $0<p<1$. 
	
	 Peetre (\cite{peetre}) showed that in the definition of $W^{1,p}(I)$, we can work with $C^1_+$ functions instead of $C^1$ ones. Here we say that a function $f$ is $C^1_+(0,1)$ if it is continuous, has continuous derivative except in a finite set, where $f$ admits left and right derivatives. Analysing the proof, it is clear that we can do the same if we are dealing only with homeomorphisms. That is, if $f$ is a $C^1_+$ homeomorphism then there exists a sequence of $C^1$ homeomorphism converging in $W^{1,p}(I)$ to $f$. In fact that sequence also converge uniformly.

In Section \ref{d=1}, we will find, for some $C^1$ homeomorphisms $f$, all the pairs $(f,F)$ that belongs to $\mathcal{M}^{1,p}(I)$.

	\begin{maintheorem}\label{TheoremA}			
	Let  $0<p<1$, $f$ be a $C^1$ homeomorphism of a bounded open interval $I$ such that and  $f',\big(f^{-1}\big)'\in L^r$, for some $r>1$. 
	
	If $F\in L^p(I)$ then $(f,F)\in \mpi$ if and only if $0\leq \frac{F}{f'}\leq 1$. 
	
	In particular, if $(f,F)\in\mpi$ then $F\in L^r(I)$ and $(f,F)\in\mathcal{M}^{1,q}(I)$, for all $0<q<1$.
\end{maintheorem}

Section \ref{technical} will consist on proofs of some technical results to be use used in Section \ref{dgeq2}. In particular, we will prove that given $r>s>0$ and $H\in SO(d)$ there exists  an orientation and volume preserving $C^{\infty}$ diffeomorphism $G:\R^d\rightarrow\R^d$ that is equal to $H$ in the disk $D(0,s)$ and to the identity outside the disk $D(0,r)$.

In Section \ref{dgeq2}, we will find a family of elements $H\in L^p(\Omega)^{d^2}$ such that $(I_d,H)\in\mlp{\Omega}$. We will also analyse the case where, in the place of the identity function, we consider a volume preserving $C^1$ homeomorphism. 

\begin{maintheorem}\label{TheoremB} Let $d\geq 2$, $0<p<1$, $\Omega$ a bounded  open subset of $\R^d$. If $H:\Omega\longrightarrow SO(d)$ is Riemann integrable then $(I_d,H)\in\mlp{\Omega}$.
\end{maintheorem}

 We note that the usual methods of approximating functions, using mollifiers and convolution product, cannot be used here since all the functions in question are homeomorphisms, and still preserving the measure in Section \ref{dgeq2}. 

\section{The proof of Theorem \ref{TheoremA}}\label{d=1}

\begin{definition} Let $I$ be a bounded interval of $\R$ and $0<p<1$. We define $\mpi$ as  the set formed by the pairs $(f, F)\in L^p(I)\times L^p(I)$, admitting a sequence $(f_n)_n$ of $C^1$ homeomorphisms  of $I$, converging uniformly to $f$ and such that $(f'_n)_n$ converge to $F$ in $L^p(I)$.
\end{definition}

We start with an example, adapted  from \cite{peetre} (see also \cite{haakan}), that shows that $(I_d,0)\in\mathcal{M}^{1,p}(0,1)$, for all $0<p<1$. This example will be crucial in the proof of  Theorem \ref{TheoremA}.

\begin{example} \label{peetre2}
		For $n\in\N$, let $a_n,b_n>0$ with $a_n+b_n=\frac{1}{n}$ and consider $f_n:[0,\frac{1}{n}]\rightarrow [0,\frac{1}{n}]$ defined by	
			\begin{equation*}
			f_n(x)=
			\left\{\begin{array}{ll}
				g_n(x) & \text{if $x\leq a_n$}\\[2mm]
				\frac{a_n}{b_n}\,g_n(\frac{a_n}{b_n}(x-a_n))+b_n & \text{if $x> a_n$,}
			\end{array}\right.
		\end{equation*}
		where $g_n:[0,a_n]\rightarrow [0,b_n]$ is a strictly increasing $C^1$ homeomorphism  admitting right derivative in $0$ and left derivative in $a_n$ (for example $g_n(x)=\frac{b_n}{a_n}\,x$).
		
		We extend $f_n$ to a (increasing) homeomorphism from $[0,1]$ to $[0,1]$ by defining $f_n(\frac{k}{n}+x)=f_n(x) +\frac{k}{n}$, if $x\in \left]\frac{k}{n},\frac{k+1}{n}\right]$, with $k\in\{0,1\ldots,n-1\}$ (see Figure \ref{figura}). 		
		Notice that $f_n\in C^1_+(0,1)$ has continuous derivative except  possibly in the points $\frac{k}{n}+sa_n$, with $k=1,\ldots,n-1$ and $s=0,1$, where $f$ admits left and right derivatives. 
		
		Of course we can choose $g_n$ in such a way that $f_n$ is $C^k$ for any given $k\in\N\cup\{\infty\}$. For example, if we choose $g_n$ to be a $C^\infty$ function such that $g_n^{(k)}(0)=g_n^{(k)}(a_n)=0$ for all $k\in\N$ (a ``jump function'') then $f_n$ is $C^\infty$. An explicit example of such a function is defined by
		\begin{equation*}
			g_n(x)=b_n\frac{ e^{-\frac{a_n}{x}} }{ e^{-\frac{a_n}{x}}+e^{-\frac{a_n}{a_n-x}}}.
		\end{equation*}

		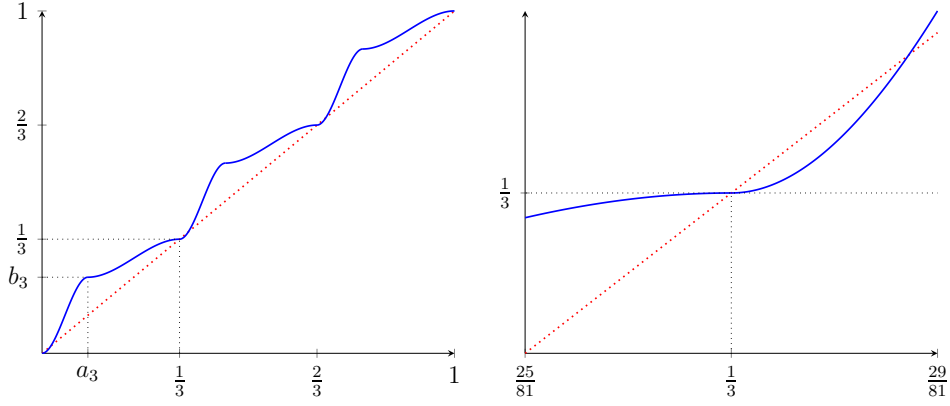
\begin{figure}[htb]
			\centering	\begin{tikzpicture}[xscale=0.8,yscale=0.8]
				]
				\begin{axis}[xtick={1/9,1/3,2/3,1},xticklabels={$a_3$,$\frac{1}{3}$,$\frac{2}{3}$, $1$},ytick={2/9,1/3,2/3,1},yticklabels={$b_3$,$\frac{1}{3}$,$\frac{2}{3}$, $1$},
					axis lines = left,
					]
					\addplot [ thick, dotted,
					domain=0:1, 
					samples=100, 
					color=red,
					]
					{x};
					\addplot [thick,
					domain=0:1/9, 
					samples=100, 
					color=blue,
					]
					{54*x^2*(-6*x+1)};
					\addplot [ thick,
					domain=1/9:1/3, 
					samples=100, 
					color=blue,
					]
					{-9/4*x*(3*x-1)^2+1/3};
					\addplot [thick,
					domain=1/3:4/9, 
					samples=100, 
					color=blue,
					]
					{6*(3*x-1)^2*(-6*x+3)+1/3};
					\addplot [ thick,
					domain=4/9:2/3, 
					samples=100, 
					color=blue,
					]
					{-9/4*(x-1/3)*(3*x-2)^2+2/3};
					\addplot [thick,
					domain=2/3:7/9, 
					samples=100, 
					color=blue,
					]
					{6*(3*x-2)^2*(-6*x+5)+2/3};
					\addplot [thick,
					domain=7/9:1, 
					samples=100, 
					color=blue,
					]
					{-9/4*(x-2/3)*(3*x-3)^2+1};
					\draw[dotted] (axis cs:1/9,0) -- (axis cs:1/9, 2/9);
					\draw[dotted] (axis cs:0,2/9) -- (axis cs:1/9, 2/9);
					\draw[dotted] (axis cs:1/3,0) -- (axis cs:1/3, 1/3);
					\draw[dotted] (axis cs:0,1/3) -- (axis cs:1/3, 1/3);
				\end{axis}
			\end{tikzpicture}
			\ 
			\centering	\begin{tikzpicture}[xscale=0.8,yscale=0.8]
				]
				\begin{axis}[xtick={25/81,1/3,29/81},xticklabels={$\frac{25}{81}$,$\frac{1}{3}$, $\frac{29}{81}$},ytick={2/9,1/3,4/9},yticklabels={$\frac{2}{9}$,$\frac{1}{3}$, $\frac{4}{9}$},
					axis lines = left,
					]
					\addplot [ thick, dotted,
					domain=25/81:29/81, 
					samples=100, 
					color=red,
					]
					{x};
					\addplot [ thick,
					domain=25/81:1/3, 
					samples=100, 
					color=blue,
					]
					{-9/4*x*(3*x-1)^2+1/3};
					\addplot [thick,
					domain=1/3:29/81, 
					samples=100, 
					color=blue,
					]
					{6*(3*x-1)^2*(-6*x+3)+1/3};
					\draw[dotted] (axis cs:1/9,0) -- (axis cs:1/9, 2/9);
					\draw[dotted] (axis cs:0,2/9) -- (axis cs:1/9, 2/9);
					\draw[dotted] (axis cs:1/3,0) -- (axis cs:1/3, 1/3);
					\draw[dotted] (axis cs:0,1/3) -- (axis cs:29/81, 1/3);
				\end{axis}
			\end{tikzpicture}
			\caption{In the left, the graphic of $f_3$ (with $a_3=\frac{1}{9}$ and $g_3(x)=54x^2(1-6x)$). In this case $f_3$ is a $C^1$ function. In the right, a zoom of the graphic of $f_3$ around $x=\frac{1}{3}$.}
			\label{figura}
		\end{figure}

		Since $\|f_n-I_d\|_\infty\leq \frac{1}{n}$ then $(f_n)_n$ converges uniformly to $I_d$. On the other hand, using \eqref{pqmenor}, 
		\begin{equation*}
			\|g_n'\|_p^p\leq a_n^{1-p}\|g_n'\|_1^p= a_n^{1-p}\int_0^{a_n}g_n'(x)\,dx=a_n^{1-p}b_n^p,
		\end{equation*}
		and then
		\begin{align*}
			\|f'_n\|^p_p & =n\left(\int_0^{a_n}g'_n(x)^pdx
			+\int_{a_n}^{\frac{1}{n}}h'_n(x)^pdx\right)\\
			& = n\left(	\int_0^{a_n}g'_n(x)^pdx
			+b_n^{1-2p}a_n^{2p-1}\int_0^{a_n}g'_n(y)^pdy\right)\\
			& \leq 
			n(a_n^{1-p}b_n^p+a_n^pb_n^{1-p})=(na_n)^{1-p}(nb_n)^p+(nb_n)^{1-p}(na_n)^p.
		\end{align*}
		
		Choosing $a_n$ such that $\lim_n na_n=0$, which implies $\lim_n nb_n=1$, then $(f_n')_n$ converges to $0$ in $L^p(0,1)$.
\end{example}

\begin{remark}\label{Idc}
	If $c\in[0,1]$ then $(I_d,c)\in {\mathcal M}^{1,p}(0,1)$. To see this, consider the sequence $h_n=\big(c\,I_d+(1-c) f_n\big)_n$, where $(f_n)_n$ is a sequence of $C^1$ functions as defined in Example \ref{peetre2}. It is clear that, for all $n\in\N$, $h_n$ is a $C^1$ homeomorphism of $[0,1]$, as $h_n'\geq 0$ and $h_n(0)=0$ and $h_n(1)=1$. Of course we can defined $f_n$ in such a way that $h_n$ is $C^\infty$.
\end{remark}

The next two results prove Theorem \ref{TheoremA} in the particular case where $f=I_d$.

\begin{lemma}\label{0h1}
	If $(I_d,H)\in\mathcal{M}^{1,p}(0,1)$ then $0\leq H\leq 1$.
\end{lemma}
\begin{proof}
Let $(f_n)_n$ be a sequence of $C^1$ homeomorphisms of $[0,1]$ converging uniformly to the identity, such that $(f'_n)_n$ converges in $L^p(0,1)$, and then almost everywhere, to $H$.
As $f_n(0),f_n(1)\in\{0,1\}$, $\lim_nf_n(0)=0$, $\lim_nf_n(1)=1$ and $f_n$ is monotone, then, for almost every $n$, $f_n$ is increasing, which implies $f_n'\geq 0$ and so $H\geq 0$.

Suppose now that $H\not\leq 1$. As $H$ is the supremum of the simple functions below $H$, there exist $b>1$ and $E$, a measurable subset of $[0,1]$, such that $\lambda(E)>0$ and  $H\geq b$ in $E$. For $\delta>0$, to be defined depending on $\lambda(E)$ and $b$, let $U$ be an open set containing $E$ such that $\lambda(U\setminus E)<\delta$. Let $(U_i)_i$ be a sequence of open disjoint intervals whose union is $U$. As $\lambda(E)=\lambda(U\cap E)= \sum_{i=1}^\infty\lambda(U_i\cap E)$, there exists $N\in\N$ such that, if $V=\cup_{i=1}^NU_i$ and $F=E\cap V$,
\begin{equation*}
	F\subseteq V,\ \lambda(F)>\tfrac{1}{2}\lambda(E),\ \lambda(V\setminus F)<\delta.
\end{equation*}

Consider $0<t<1$ such that $t^2b-1>0$ and
\begin{equation*}
	A_n=\big\{x\in F: f'_k(x)\geq tb,\text{ for all $k\geq n$}\big\}.
\end{equation*} 
As the sequence $(A_n)_n$ is increasing and, by hypothesis on $(f'_n)_n$, almost every element of $F$ is in  $\cup_n A_n$ then, denoting the Lebesgue measure by $\lambda$, we have
\begin{equation*}
	\lim_n\lambda(A_n)=\lambda(F).
\end{equation*}

For $\eps>0$, consider $n\in\N$ such that $\lambda(A_n)\geq t\lambda(F)$ and $|f_n(x)-x|\leq \eps$, for all $x\in [0,1]$. Then, if $U_i=]\alpha_i,\beta_i[$ for $i=1,\ldots,N$,
\begin{align*}
	\int_{V}f_n' & \geq\int_{A_n}f'_n\geq t^2b\lambda(F),\\	
	\int_{V}f_n' &=\sum_{i=1}^N\big(f_n(\beta_i)-f_n(\alpha_i)\big)\\
	 & \leq \sum_{i=1}^N\big(\beta_i-\alpha_i+2\eps\big)=\lambda(V)+2N\eps\leq \lambda(F)+\delta+2N\eps.
\end{align*}
From this, we obtain $0<(t^2b-1)\lambda(F)\leq \delta+2N\eps$, for all $\eps>0$, and then $0<(t^2b-1)\lambda(F)\leq \delta$ which is a contradiction, by choosing $\delta<\frac{1}{2}(t^2b-1)\lambda(E)$ and recalling that $\tfrac{1}{2}\lambda(E)<\lambda(F)$.
 \end{proof}

 We are now in the conditions to replace the constant functions, referred to in Remark \ref{Idc}, by step functions $H$, with $0\leq H\leq 1$. By density the same will be true if $H$ is any $L^p$ function with $0\leq H\leq 1$.

\begin{proposition} \label{l1} If $0<p<1$, $I$ a non-empty  open bounded interval of $\R$ and $H\in L^p(I)$ then $(I_d,H)\in {\mathcal M}^{1,p}(I)$ if and only if $0\leq H\leq 1$.
	
In particular, if $(I_d,H)\in\mpi$ then $H\in L^\infty(I)$ and $(I_d,H)\in\mathcal{M}^{1,q}(I)$, for  $0<q<1$.  
\end{proposition}
\begin{proof} Without loss of generality, we will do the proof for $I=(0,1)$. Taking into account Lemma \ref{0h1} we are left to prove that, if
	$0\leq H\leq 1$ then $(I_d,H)\in {\mathcal M}^{1,p}(I)$. 
	
Suppose first that $H$ is a step function such that $0\leq H\leq 1$ and consider $N\in\N$, $0=a_0<a_1<\cdots <a_{N-1}<a_N=1$ a partition of $[0,1]$ and $c_1,\ldots,c_N\in [0,1]$ such that $H(x)=c_i$, for all $x\in [a_{i-1},a_i[$.
	
	Consider the sequence $(\phi_n)_n$ defined by
	\begin{equation*}
		\forall x\in [a_{i-1},a_i]\quad \phi_n(x)=(a_i-a_{i-1})\,h_{n,i}\Big(\frac{x-a_{i-1}}{a_i-a_{i-1}}\Big)+a_{i-1}
	\end{equation*}
	where, for $i=1,\ldots,N$,  $(h_{n,i})_n$ is the sequence of $C^\infty$ homeomorphisms of $(0,1)$ described in Remark \ref{Idc} for $c=c_i$ (see Figure \ref{figura_derivada} for an example of the derivative of these functions).  It is clear that, for all $n\in\N$, $\phi_n$ is a $C^1_+$ homeomorphism of $(0,1)$, as $\phi_n$ has derivatives (of all orders) except in $a_i$ ($i=0,1,\ldots, N$), were there are only lateral derivatives.
	
	To see that $(\phi_n)_n$ converges uniformly to the identity, just notice that
	\begin{align*}
		\big\|\phi_n-I_d\big\|& =\max_{i=1,\ldots,N}\left\{ (a_i-a_{i-1})\left[\max_{x\in[a_{i-1},a_i]}\left|h_{n,i}\Big(\frac{x-a_{i-1}}{a_i-a_{i-1}}\Big)-\frac{x-a_{i-1}}{a_i-a_{i-1}}\right|\right]\right\}\\
		& =\max_{i=1,\ldots,N}\left\{ (a_i-a_{i-1})\left[\max_{y\in[0,1]}\left|h_{n,i}(y)-y\right|\right]\right\}\\
		 &  =\max_{i=1,\ldots,N}\left\{ (a_i-a_{i-1})\right\}\,\|h_{n,i}-I_d\|_\infty.
	\end{align*}
	
On the other hand, $(\phi_n')_n$ converges to $H$ in $L^p(I)$ as
\begin{align*}
	\|\phi_n'-H|_p^p & =\sum_{i=1}^N\int_{a_{i-1}}^{a_i}\left| h_{n,i}'\Big(\frac{x-a_{i-1}}{a_i-a_{i-1}}\Big)-c_i\right|^pdx\\
	& =\sum_{i=1}^N\int_0^1\left|h_{n,i}'(y)-c_i\right|^p(a_i-a_{i-1})\,dy\\
		& \leq N\,\|h_{n,i}'-c_i\|_p^p.
\end{align*}	

Consider now the general case. Given $n\in\N$, consider a step function $H_n:I\rightarrow I$ such that $\|H-H_n\|_1\leq \frac{1}{n}$. Using the above, let $f_n$ be a $C^1_+$ homeomorphism such that $\|f_n-I_d\|_p\leq \frac{1}{n}$ and $\|Df_n-H_n\|_p\leq \frac{1}{n}$. Then, using \eqref{desigualdadep} and \eqref{pqmenor},
	\begin{equation*}
		\|Df_n-H\|_p^p\leq \|Df_n-H_n\|_p^p+ \|H_n-H\|_p^p\leq  \|Df_n-H_n\|_p^p+\|H_n-H\|_1^p,
	\end{equation*}
	from where we conclude that $(I_d,H)\in {\mathcal M}^{1,p}(I)$.
	\end{proof}

	\begin{figure}[H]
	\begin{tikzpicture}[>=stealth]
	\begin{axis}[xtick={2/5,3/5,1},xticklabels={$\frac{2}{5}$,$\frac{3}{5}$, $1$},ytick={7/2,19/4,31/6},yticklabels={$\frac{7}{2}$,$\frac{11}{4}$, $\frac{31}{6}$},width=\textwidth,    
		height=0.75\textwidth,
		domain=0:1,xmin=0,xmax=1.1,
		ymin=0,ymax=5.7,
		axis x line=middle,
		axis y line=middle,
		axis line style=->,
		xlabel={$x$},
		ylabel={$y$},
		]
		\addplot[no marks,red] expression[domain=0:2/5,samples=100]{1/4}; 
		\addplot[no marks,red] expression[domain=2/5:3/5,samples=100]{1/2}; 
		\addplot[no marks,red] expression[domain=3/5:1,samples=100]{1/6}; 
		\addplot[no marks,blue] expression[domain=0:2/125,samples=100]{-70312.5*x^2+1125.*x+.25}; 
		\addplot[no marks,blue] expression[domain=2/125:2/25,samples=100]{-274.66*x^2-.10156+26.367*x}; 
		\addplot[no marks,blue]			expression[domain=2/25:12/125,samples=100]{-539.75-70312.*x^2+12375.*x}; 
		\addplot[no marks,blue]
		expression[domain=12/125:4/25,samples=100]{-274.66*x^2-3.9688+70.312*x}; 
		\addplot[no marks,blue]
		expression[domain=4/25:22/125,samples=100]{-1979.8-70312.*x^2+23625.*x}; 
		\addplot[no marks,blue]
		expression[domain=22/125:6/25,samples=100]{-274.66*x^2-11.352+114.26*x }; 
		\addplot[no marks,blue]
		expression[domain=6/25:32/125,samples=100]{-4319.750000-70312.50000*x^2+34875.*x}; 
		\addplot[no marks,blue]
		expression[domain=32/125:8/25,samples=100]{-274.66*x^2-22.250+158.20*x }; 
		\addplot[no marks,blue]
		expression[domain=8/25:42/125,samples=100]{-7559.750000-70312.50000*x^2+46125.*x}; 
		\addplot[no marks,blue]
		expression[domain=42/125:2/5,samples=100]{ -274.66*x^2-36.664+202.15*x}; 
		\addplot[no marks,blue]
		expression[domain=2/5:51/125,samples=100]{-187500.*x^2-30599.50000+151500.*x};
	\addplot[no marks,blue]
	expression[domain=51/125:11/25,samples=100]{-732.42*x^2+621.09*x-130.98 }; 
	\addplot[no marks,blue]
	expression[domain=11/25:56/125,samples=100]{-187500.*x^2-36959.50000+166500.*x}; 
	\addplot[no marks,blue]
	expression[domain=56/125:12/25,samples=100]{-732.42*x^2-157.+679.69*x }; 
	\addplot[no marks,blue]
	expression[domain=12/25:61/125,samples=100]{-187500.*x^2-43919.50000+181500.*x}; 
	\addplot[no marks,blue]
	expression[domain=61/125:13/25,samples=100]{-732.42*x^2-185.36+738.28*x}; 
	\addplot[no marks,blue]
	expression[domain=13/25:66/125,samples=100]{-187500.*x^2-51479.50000+196500.*x}; 
	\addplot[no marks,blue]
	expression[domain=66/125:14/25,samples=100]{-732.42*x^2-216.06+796.88*x }; 
	\addplot[no marks,blue]
	expression[domain=14/25:71/125,samples=100]{-187500.*x^2-59639.50000+211500.*x }; 
	\addplot[no marks,blue]
	expression[domain=71/125:3/5,samples=100]{ -732.42*x^2+855.47*x-249.11}; 
	\addplot[no marks,blue]
	expression[domain=3/5:77/125,samples=100]{-78125.*x^2-28874.83333+95000.*x }; 
	\addplot[no marks,blue]
	expression[domain=77/125:17/25,samples=100]{-305.18*x^2-127.67+395.51*x}; 
	\addplot[no marks,blue]
	expression[domain=17/25:87/125,samples=100]{-36974.83333-78125.*x^2+107500.*x }; 
	\addplot[no marks,blue]
	expression[domain=87/125:19/25,samples=100]{-161.26-305.18*x^2+444.34*x }; 
	\addplot[no marks,blue]
	expression[domain=19/25:97/125,samples=100]{-78125.*x^2+120000.*x-46074.83333 }; 
	\addplot[no marks,blue]
	expression[domain=97/125:21/25,samples=100]{-305.18*x^2-198.76+493.16*x }; 
	\addplot[no marks,blue]
	expression[domain=21/25:107/125,samples=100]{ -78125.*x^2+132500.*x-56174.83333}; 
	\addplot[no marks,blue]
	expression[domain=107/125:23/25,samples=100]{-305.18*x^2-240.17+541.99*x }; 
	\addplot[no marks,blue]
	expression[domain=23/25:117/125,samples=100]{ -78125.*x^2-67274.83333+145000.*x}; 
	\addplot[no marks,blue]
	expression[domain=117/125:1,samples=100]{-305.18*x^2-285.48+590.82*x};
\end{axis}
\end{tikzpicture}
	\caption{In blue the graphic of $\phi_5'$ as defined in the proof of Proposition \ref{l1}, relatively to the step function in red $\frac{1}{4}\chi_{[0,\frac{2}{5}]}+\frac{1}{2}\chi_{[\frac{2}{5},\frac{3}{5}]}+\frac{1}{6}\chi_{[\frac{3}{5},1]}$.}
	\label{figura_derivada}
\end{figure}
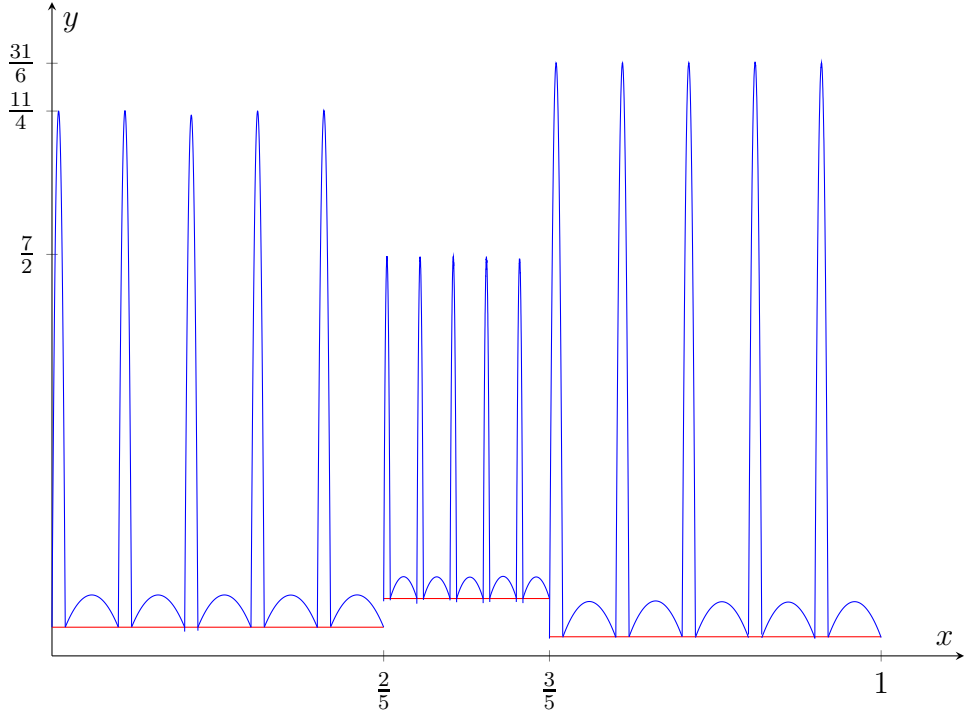

The next result is a tool that will enable us to go from the  characterization of  the pairs $(I_d,H)\in\mpi$ to the case described in Theorem \ref{TheoremA}.
		\begin{proposition}			
		Let $r>1$, $g,h$ be $C^1_+$ homeomorphisms of a bounded open interval $I$ such that 
		$\varphi'\in L^r$, where $\varphi=h^{-1}\circ g$. 
		
		If $0<q<1$ and $(g,G)\in \mathcal{M}^{1,\frac{qr}{r-1+q}}(I)$ then $(h,H)\in\mathcal{M}^{1,q}(I)$, where $H=\big(G\circ \varphi^{-1}\big)\cdot \big(\varphi^{-1}\big)'$.
	\end{proposition}
	\begin{proof} First notice that $q<\frac{qr}{r-1+q}<1$.
		Consider a sequence $(g_n)_n$ of $C^1$ homeomorphisms of $I$ converging uniformly to $g$ and such that  $(g_n')_n$ converges to $G$ in $L^q$. Then the sequence $(g_n\circ \varphi^{-1})_n$ converges uniformly to $h$ as
		\begin{equation*}
			\|g_n\circ \varphi^{-1}-h\|_\infty=\sup_{x\in I}|g_n((g^{-1}(h(x)))-g(g^{-1}(h(x)))|
			=\|g_n-g\|_\infty.
		\end{equation*}
		On the other hand,
		\begin{align*}
			\|(g_n\circ \varphi^{-1})'-H\|_q^q & =\int_I\big|g_n'(\varphi^{-1}(x))\big(\varphi^{-1}\big)'(x)-H(x)\big|^qdx\\
			& =\int_I\big|g_n'(y)\big(\varphi'(y)\big)^{-1}-H(\varphi(y))\big|^q |\varphi'(y)|\,dy\\
			& =\int_I\big|g_n'(y)-G(y)\big|^q \big|\varphi'(y)\big|^{1-q}\,dy.
		\end{align*}
		
		Using H\"older's inequality with the conjugate exponents $\frac{r}{r-1+q}$ and $\frac{r}{1-q}$, we obtain
		\begin{align*}
			\|(g_n\circ \varphi^{-1})'-H\|_q^q & \leq\||\varphi'|^{1-q}\|_{\frac{r}{1-q}}\,\|\big|g_n'-G\big|^q\|_{\frac{r}{r-1+q}}
		\end{align*}
		or, equivalently,
		\begin{align*}
		\|(g_n\circ \varphi^{-1})'-H\|_q^q & \leq\|\varphi'\|_r^{1-q}\,\|g_n'-G\|_{\frac{qr}{r-1+q}}^q,
		\end{align*}
		completing the proof that $(h,H)\in\mathcal{M}^{1,q}(I)$.
	\end{proof}

Applying this result when $f$ or $g$ are the identity and recalling Proposition \ref{l1} we can now prove Theorem \ref{TheoremA}.

\begin{proof} (of Theorem \ref{TheoremA})
 Let $(f,F)\in \mpi$. Using the previous proposition, with  $q=\frac{p(r-1)}{r-p}$, $g=f$ and $h=I_d$, noticing that $0<q<p$ and $\frac{qr}{r-1+q}=p$, we obtain $\big(I_d,(F\circ f^{-1})\cdot (f^{-1})'\big)\in\mathcal{M}^{1,q}(I)$. By Proposition \ref{l1},
 \begin{equation*}
  0\leq (F\circ f^{-1})\cdot (f^{-1})'\leq 1,
 \end{equation*}
 and then, for $x\in I$, $ 0\leq (F\circ f^{-1})(f(x))\cdot (f^{-1})'(f(x))\leq 1,$ or, equivalently  $0\leq \frac{F(x)}{f'(x)}\leq 1$.

 Suppose now that  $0\leq \frac{F}{f'}\leq 1$. Working like above we conclude that $ 0\leq (F\circ f^{-1})\cdot (f^{-1})'\leq 1$ and then, by Proposition \ref{l1}, 
 \begin{equation*}
 \big(I_d,(F\circ f^{-1})\cdot (f^{-1})'\big)\in \mathcal{M}^{1,\frac{pr}{r-1+p}}.
 \end{equation*}
 
 Using the last proposition again, for $g=I_d$, $h=f$ and $q=p$, we obtain 
  \begin{equation*}
 	\Big(f,\big[(F\circ f^{-1})\cdot (f^{-1})'\big]\circ f)\cdot f'\Big)\in \mathcal{M}^{1,p}
 \end{equation*}
or, equivalently, $(f,F)\in\mpi$.
\end{proof}

\section{Some technical results}\label{technical} 
 We start by proving some auxiliary results  in order to prove the main result of the section, Proposition \ref{peetrerotacao2}.
\begin{lemma}\label{rotacao}
	Let $d\geq 2$, $1\leq m\leq \frac{d}{2}$,  $h=(h_1,\ldots,h_m):\R\rightarrow\R^s$  a $C^k$ function ($k\in\N\cup\{\infty\}$) and  $\alpha=(\alpha_1,\ldots,\alpha_m): \R^d \longrightarrow  \R^s$ defined by $\alpha(x)=h(x_1^2+\cdots+x_d^2)$.
	
	Then the function $F^h=(F_1^h,\ldots,F_d^h): \R^d \rightarrow \R^d$ such that
	\begin{equation*}
		F_t^h(x) = \left\{\begin{array}{ll}
			x_{2j-1}\cos(\alpha_j(x))-x_{2j}\sin(\alpha_j(x)), &  \text{if $t=2j-1$ and $j\leq m$}\\[2mm]
			x_{2j-1}\sin(\alpha_j(x))+x_{2j}\cos(\alpha_j(x)), &  \text{if $t=2j$ and $j\leq m$}\\[2mm]
			x_t, & \text{if $t>2m$}
		\end{array}\right.
	\end{equation*}
	is a volume and orientation preserving diffeomorphism, with the same differential regularity as $h$.
\end{lemma}
\begin{proof}
	First notice that, for $j\leq m$, $	F_{2j-1}^h(x)^2+F_{2j}^h(x)^2=x_{2j-1}^2+x_{2j}^2$,
	and then
	\begin{equation*}
		F_1^h(x)^2+F_2^h(x)^2+\cdots +F_d^h(x)^2=x_1^2+x_2^2+\cdots+x_d^2.
	\end{equation*}
	
	From the last equality, it is easy to conclude that $F^h$ is a diffeomorphism whose inverse is $F^{-h}$. 
	
	To conclude the proof, we only need to show that $\det{JF(x)}=1$. Of course we can suppose that $2m=d$.
	
	Notice that the $k^{th}$ column of the $JF^h(x)$ is of the forma $A_k+B_k$, where $A_k=(a_{i,k})_{i=1,\ldots,d}$ and $B_k=(b_{i,k})_{i=1,\ldots,d}$, with
	\begin{equation*}
		a_{i,k}=\left\{\begin{array}{rll}
			\cos(\alpha_j(x)) & \text{if $i=2j-1$} & \text{ and $k=i$}\\
			-\sin(\alpha_j(x)) & \text{if $i=2j-1$} & \text{ and $k=i+1$}\\[2mm]
			\sin(\alpha_j(x)) & \text{if $i=2j$} & \text{ and $k=i-i$}\\
			\cos(\alpha_j(x)) & \text{if $i=2j$} & \text{ and $k=i$}\\[2mm]
			0 & \text{otherwise}, &
		\end{array}\right.		
	\end{equation*}

	\begin{equation*}
		b_{i,k}=\left\{\begin{array}{rl}
			\Big[-2x_ix_k\sin(\alpha_j(x))-2x_{i+1}x_k\cos(\alpha_j(x))\Big]h_j'(\|x\|_2^2) & \text{if $i=2j-1$}\\[2mm]
			\Big[2x_{i-1}x_k\cos(\alpha_j(x))-2x_ix_k\sin(\alpha_j(x))\Big]h_j'(\|x\|_2^2) & \text{if $i=2j$.}
		\end{array}\right.		
	\end{equation*}
	
	It is clear that
	\begin{equation}\label{elementar}
		\forall k_1,k_2=1,\ldots d\quad x_{k_2}B_{k_1}=x_{k_1}B_{k_2}.
	\end{equation}
	
	Using the multilinear property of the determinant,
	\begin{align*}
		\det{JF^h(x)}=&\det{\left(A_1+B_1\quad A_2+B_2\quad \ldots\quad A_{d-1}+B_{d-1}\quad A_d+B_d\right)}
	\end{align*}
	is equal to 
	\begin{equation*}
		\sum_{I\subseteq\{1,\ldots,d\}}\det{\left(C_1\ C_2\ \ldots\ C_d\right)},\ \text{where $C_k=A_k$, if $k\in I$ and $C_k=B_k$, otherwise.}
	\end{equation*}
	
	Using property \eqref{elementar}, the summation above is equal to 
	\begin{align*}
		\det{\left(A_1\ A_2\ \ldots\ A_d\right)}+\\
		\det{\left(B_1\ A_2\ \ldots\ A_d\right)}+\det{\left(A_1\ B_2\ A_3\ldots\ A_d\right)}+\\
		\det{\left(A_1\ A_2\ B_3\ A_4 \ldots\ A_d\right)}+\det{\left(A_1\ A_2\ A_3\ B_4\ A_5\ \ldots\ A_d\right)}+\\
		+\cdots+\\
		\det{\left(A_1\  \ldots\ A_{d-2}\ B_{d-1}\ A_d\right)}+\det{\left(A_1\ A_2\ \ldots\ A_{d-1}\ B_d\right)}.
	\end{align*}

	The first matrix in this sum is the block matrix 
	\begin{equation*}
		\left( \begin{array}{cccc}
			X_1 & 0 & \ldots & 0  \\
			0 & X_2 & \dots & 0  \\
			\vdots & \vdots & \ddots & \vdots\\
			0 & 0 & \dots & X_m
		\end{array} \right),\quad\text{with}\quad 	X_t=\left( \begin{array}{ll}
			a_{2t-1,2t-1} & a_{2t-1,2t}\\ 
			a_{2t,2t-1} & a_{2t,2t}
		\end{array} \right),
	\end{equation*}
	whose determinant is equal to $1$. In what concern to the others matrices, all them are of the form
	\begin{equation*}
		\left( \begin{array}{cccc||lccc}
			X_1 & 0 & \ldots & 0& Y_1 & 0 &\ldots & 0  \\
			0 & X_2 & \dots & 0 & Y_2 & 0  & \dots & 0  \\
			\vdots & \vdots & \ddots & \vdots & \vdots  & \vdots & \ddots & \vdots \\
			0 & 0 & \dots & X_{j-1} & 0 & 0  & \dots & 0  \\[1mm]
			\hline\hline
			&&&&&&&\\[-2mm]
			0 & 0 & \ldots & 0 & Y_{j}  & 0 & \ldots & 0 \\
			0 & 0 & \ldots &0 & Y_{j+1}  & X_{j+1} & \ldots &  0 \\
			\vdots& \vdots & \ldots & \vdots & \vdots  & \vdots & \ddots &0  \\
			0 & 0 & \ldots & 0 &  Y_s  & 0 & \ldots & X_m \\
		\end{array} \right),
	\end{equation*}
	where
	\begin{equation*}
		Y_t=\left\{\begin{array}{l}
			Y_t^1=\left( \begin{array}{rr}
				a_{2t-1,2t-1} & b_{2t-1,2t}\\ 
				a_{2t,2t-1} & b_{2t,2t}
			\end{array} \right)\\
			\text{or}\\
			Y_t^2=\left( \begin{array}{ll}
				b_{2t-1,2t-1} & a_{2t-1,2t}\\ 
				b_{2t,2t-1} & a_{2t,2t}
			\end{array} \right).
		\end{array}\right.
	\end{equation*}
	
	The determinants of those matrices are
	\begin{equation*}
		\det{	\left( \begin{array}{cccc}
				X_1 & 0 & \ldots & 0  \\
				0 & X_2 & \dots & 0  \\
				\vdots & \vdots & \ddots & \vdots\\
				0 & 0 & \dots & X_m
			\end{array} \right)}\times \det{
			\left( \begin{array}{lccc}
				Y_{j}  & 0 & \ldots & 0 \\
				Y_{j+1}  & X_{j+1} & \ldots &  0 \\
				\vdots & \vdots & \ddots & 0\\
				Y_s  & 0 & \ldots &X_m \\
			\end{array} \right)=\det{Y_j.}
		}
	\end{equation*}
	
	Then 
	\begin{align*}
		\det{JF^h(x)}=1+\sum_{j=1}^s\left(\det{Y_j^1}+\det{Y_j^2}\right)
	\end{align*}
	and the conclusion follows, as one can see that $\det{Y_j^1}+\det{Y_j^2}=0$.
\end{proof}

We will denote by $J$, a ``jump function'' that establishes a $C^\infty$ transition from $1$ to $0$, that is, $J:\R\rightarrow \R$ equal to $1$ in $]-\infty,0]$, strictly decreasing in $[0,1[$ and equal to $0$ in $[1,+\infty[$.

Of course all the derivative of $J$ are bounded.
An explicit example of such a function $J$ is defined in $(0,1)$ by 
$$J(t)=\frac{ \,e^{\frac{1}{t-1}}}{e^{\frac{1}{t-1}}+e^{-\frac{1}{t}}}.$$

As a curiosity, one can verify that $\max_\R |J'|=|J'(\frac{1}{2})|=2$.

\begin{proposition}\label{peetrerotacao2}
	Let $d\geq 2$, $a\in\R^d$ and $H\in SO(d)$. Then, for $r>s>0$ there exists an orientation and volume preserving $C^{\infty}$ diffeomorphism $G:\R^d\rightarrow\R^d$, such that
	\begin{eqnarray*}
		\forall x\in \R^d\quad G(x)=\left\{\begin{array}{ll}
			x & \text{if $\|
				x-a\|_2\geq r$}\\
			H(x-a)+a & \text{if $\|x-a\|_2\leq s$.}\end{array}\right.
	\end{eqnarray*}

	Moreover, there exists $C_1$, not depending on $r$ and $s$, such that
	\begin{equation*}
		\|DG-H\|_{L^p(B(a,r))}\leq C_1\,r^\frac{d}{p}\left(1-\frac{s}{r}\right)^{\frac{1-p}{p}}.
	\end{equation*}	 
\end{proposition}
\begin{proof} It is clear that we can consider $a=(0,\ldots,0)$ and that $H$ is a block diagonal matrix of the form
	\begin{equation*}
		H=\left(\begin{array}{cccc}
			H_1 & 0 & \cdots & 0\\
			0 & H_2 & \cdots & 0\\
			\vdots & \vdots & \ddots & \vdots \\
			0 & 0 & \cdots & H_k
		\end{array}\right),
	\end{equation*}
	where, for some $m\leq k$ and all $i\leq m$, there exists $\theta_i\in[0,2\pi[$ such that
	\begin{equation*}
		H_i=\left(\begin{array}{rc}
			\cos \theta_i & \sin\theta_i\\
			-\sin\theta_i & \cos\theta_i
		\end{array}\right)
	\end{equation*}
	and, for $i>m$, $H_i=(1)$.
	
	Consider $G=F^h$, given by the previous lemma, for $h=\left(h_{\theta_1},\ldots,h_{\theta_m}\right)$, where $h_{\theta_i}(t)=\theta_i\,J(\frac{t-s^2}{r^2-s^2})$.
	
	If $\|x\|_2\geq r$, then
	$\alpha(x)=h(\|x\|_2^2)=(0,\ldots,0)$, from where we obtain $G(x)=x$. On the other hand, if $\|x\|_2\leq s$, then 
	$\alpha(x)=h(\|x\|_2^2)=(\theta_1,\ldots,\theta_m)$, that is, $G(x)=H(x)$.

	For the second part, using notation introduced in the proof of Lemma \ref{rotacao}, one can see that, if $\theta=\max_i\theta_i$, 
	\begin{equation*}
		\left|\frac{\partial G_i}{\partial x_k}\right|=|a_{i,k}+b_{i,k}|\leq 1+4\theta\|J'\|_\infty\frac{r^2}{r^2-s^2}\leq (1+4\theta\|J'\|_\infty)\frac{r^2}{r^2-s^2}.
	\end{equation*}

	Notice that, as $\frac{r^2}{r^2-s^2}=\frac{r}{r+s}\cdot\frac{r}{r-s}\leq \frac{1}{2}\frac{1}{1-\frac{s}{r}}$ there exists  $C$, independent of $r$ and $s$ such that, if $i,j=1,\ldots,d$ and $w_d$ is the volume of a unitary $d$-ball, then, using the previous lemma,
	\begin{align*}
		\left	\|\frac{\partial G_i}{\partial x_j}-H_{ij}\right\|_{L^p(B(a,r))}^p & =  \left	\|\frac{\partial G_i}{\partial x_j}-H_{ij}\right\|_{L^p(B(a,r)\setminus B(a,s))}^p\\
		& \leq\left	\|\frac{\partial G_i}{\partial x_j}\right\|_{L^p(B(a,r)\setminus B(a,s))}^p+\left	\|H_{ij}\right\|_{L^p(B(a,r)\setminus B(a,s))}^p\\
		& \leq w_d(r^d-s^d)\left[C^p(1-\tfrac{s}{r})^{-p}+|H_{ij}|^p\right]\\
		& \leq w_d(r^d-s^d)(1-\tfrac{s}{r})^{-p}\left[C^p+|H_{ij}|^p\right]
	\end{align*}
	and the conclusion follows, as $r^d-s^d\leq (r-s)\,d\,r^{d-1}=d(1-\frac{s}{r})r^d$.
\end{proof}

The following result is particular case of the Vitali's covering theorem (see, for example, \cite[Corollary 2, p. 28]{Evans_Gariepy1992}).

\begin{lemma}\label{vitali}
	Let $\Omega$ be an open set of $\R^d$, $\eps>0$ and $\mathcal{V}$ be the family of all closed balls contained in $\Omega$ and of diameter at most $\eps$.
	
	Then there exists a countable disjoint subcollection of $\mathcal{V}$, $(D_n)_{n\in\N}$ such that 
	\begin{equation*}
		\lambda\left(\Omega\setminus \bigcup_{n\in\N}D_n\right)=0.
	\end{equation*}
	
	In particular, for $\delta>0$, there exists $N$ such that $\lambda\left(\Omega\setminus \bigcup_{k=1}^ND_k\right)\leq \delta$.
\end{lemma}

\section{The proof of Theorem \ref{TheoremB}}\label{dgeq2}

Analogously to Section \ref{d=1}, we will study pairs $(f,F)$ for a similar space to what we considered before, but now in higher dimensions and with the additional restriction of working with volume preserving homeomorphisms.

\begin{definition}
	If $\Omega$ is a bounded open subset of $\R^d$ and $0<p<1$, we denote by $\mlp{\Omega}$ the	subset of $L^p(\Omega)\times L^p(\Omega)^{d^2}$ whose elements are pairs $(f,F)$ admitting a sequence $(f_n)_n$ of  volume preserving $C^1$ homeomorphisms  of $\Omega$ continuous to the boundary, converging uniformly to $f$ and such that $(Df_n)_n$ converge to $F$ in $L^p(\Omega)^{d^2}$.		
\end{definition}

Notice that, if $(f,F)\in\mlp{\Omega}$ then $f$ preserves the volume as it is the uniform limit of preserving volume functions. Besides that, viewing $F$ as a matrix $d\times d$, the  determinant of $F$ is, almost everywhere, equal to  $\pm 1$. This is a consequence of the fact that the $L^p$ convergence implies almost everywhere convergence.

\begin{example}
	Consider $\Omega=(0,1)^d$ and  $(Id,0)\in L^p(\Omega)\times L^p(\Omega)^{d^2}$, where $0$ is the null matrix $d\times d$. By the considerations above, $(I_d,0)\not\in\mlp{\Omega}$, although $I_d$ is a $C^\infty$ homeomorphism that preserves the volume and there exists a sequence $(F_n)_n$ of $C^\infty$ homeomorphisms converging uniformly to $I_d$ and such that $(F_n')_n$ converges to $0$ in $L^p((0,1)^d)^{d^2}$. To see this, consider the sequence $(F_n)_n$ defined by 
	\begin{equation*}
		\begin{array}{rclc}
			F_n:&	[0,1]^d & \longrightarrow & [0,1]^d\\
			& 	(x_1,\ldots,x_d) & \longmapsto & \big(f_n(x_1),\ldots,f_n(x_d)\big),
		\end{array}
	\end{equation*}
	where $f_n$ are $C^\infty$ homeomorphisms of $(0,1)$, as defined in Example \ref{peetre2}.
\end{example}
\begin{remark}
	Using arguments as in Lemma \ref{0h1}, one can see that, if $(I_d,H)\in\mlp{\Omega}$, then $H\in L^\infty(\Omega)^d$. In fact, each entry of $H$ has modulus bounded by $1$.
\end{remark}

As a step towards proving the main result of this section, we will first do it for a particular case with simpler functions that will afterwards be able to approximate a general function.

\begin{proposition}\label{quaseconstante}
	Let $d\geq 2$, $0<p<1$, $\Omega$ a bounded open subset of $\R^d$ and suppose that $\Omega$ except a set of Lebesgue measure zero is the union of disjoint open sets $\Omega_1,\ldots,\Omega_N$.
	
	If  $H:\Omega\longrightarrow SO(d)$ is a function that is constant in each $\Omega_i$, then $(I_d,H)\in\mlp{\Omega}$.
\end{proposition}
\begin{proof}  Suppose first that $N=1$ and $\Omega=\Omega_1$.
	We will define, given $\eps>0$, a  volume preserving $C^\infty$ diffeomorphisms of $\Omega$, $f$, such that $\|f-Id\|_\infty\leq \eps$ and $\|Df-H\|_p\leq \eps$.
	
	For $\delta>0$, to be defined, consider, using Lemma \ref{vitali}, $n\in\N$, $a_1,\ldots,a_n\in \Omega$, $0<r_1,\ldots,r_n<\frac{\eps}{2}$ such that $ \left(D(a_k,r_k)\right)_{k=1}^n$ is a disjoint family of closed balls of $\Omega$ and   $\lambda\left(\Omega\setminus \cup_{k=1}^n D(a_k, r_k)\right) < \delta$. 
	
	Let $f$ be equal to the identity outside 	$ \cup_{k=1}^n D(a_k, r_k)$ and, in $D_k(a_k,r_k)$  equal to the function $G=G(a_k,H,r_k,s_k)$ given by Proposition \ref{peetrerotacao2}, for $s_k$ to be defined. Notice that $f$ is $C^\infty$ as it is $C^\infty$ in the complementary of each closed disk (recall that the closed balls are disjoints).
	
	It is clear that $f$ is a  volume preserving $C^\infty$ diffeomorphisms of $\Omega$, such that $\|f-Id\|_\infty\leq \eps$. Relatively to the control of the derivatives, we have, using the second part of Proposition \ref{peetrerotacao2} and Lemma \ref{vitali}, and recalling that 
		$Df$ is the identity outside 	$ \cup_{k=1}^n D(a_k, r_k)$,
	\begin{align*}
		\|Df-H\|_p^p & \leq \sum_{k=1}^n\|Df-H\|_{L^p(D(a_k,r_k))}^p+\|Df-H\|_{L^p(\Omega\setminus \bigcup_{k=1}^nD(a_k,r_k))}^p\\
		& \leq C_1^p\sum_{k=1}^nr_k^d\left(1-\frac{s_k}{r_k}\right)^{1-p}+\delta\,(1+\max_{i,j}|H_{i,j}|)^\frac{1}{p}
	\end{align*}
	and the conclusion follows, choosing $s_k$ close enough to $r_k$ and $\delta$ small.\vskip2mm
	
	If $N>1$, we only need to use the previous case, defining a function in each $\Omega_k$, for $k=1,\ldots,N$.
\end{proof}

\begin{proof}  (of Theorem \ref{TheoremB}) Fix a $d$-cube $K$ containing $\Omega$. For each $n\in\N$, let $\mathcal{P}_n$ be a dyadic partition of $K$ in $2^n$ $d$-cubes $K_1,\ldots,K_{2^n}$. We can suppose that, for a certain $M_n\leq 2^n$, only $K_1,\ldots,K_{M_n}$ intersects $\Omega$.  For $1\leq i\leq M_n$ choose a point $x_i$ in the interior of $K_i\cap \Omega$.
	
	Define $H_n$ such that $H_n(x)=H(x_i)$ if $x$ belongs to the interior of $K_i$, for some $i\leq M_n$. Then, by Proposition \ref{quaseconstante}, $(I_d,H_n)\in\mlp{\Omega}$ as, by definition, $H_n(x)\in SO(d)$ for almost all $x$ in $\Omega$.
	
	It is clear, as $H$ is Riemann integrable, that the sequence $\left(H_n\right)_{n\in\N}$ converge to $H$ in $L^1(\Omega)^{d^2}$.  This is true as, if $H^k$ and $H_n^k$ are the $k$-th component of $H$ and $H_n$, respectively, then
	\begin{align*}
		\int_\Omega\left|H^k(x)-H_n^k(x)\right|dx & =\sum_{i=1}^{M_n}\int_{K_i\cap\Omega}\left|H^k(x)-H_n^k(x)\right|dx\\
		&\leq \sum_{i=1}^{M_n}\left(\max_{K_i} H^k- \min_{K_i} H^k\right)\,\lambda(K_i)\\
		& =\sum_{i=1}^{M_n}\big(\max_{K_i} H^k\big)\,\lambda(K_i)- \sum_{i=1}^{M_n}\big(\min_{K_i}\,H^k\big)\,\lambda(K_i)
	\end{align*}		 
	and the conclusion follows as $H^k$ is Riemann integrable.
	
 For $n\in\N$, recalling that $(I_d,H_n)\in\mlp{\Omega}$, consider a $C^1$ homeomorphism $g_n$ such that $\|g_n-I_d\|_\infty\leq \frac{1}{n}$ and $\|Dg_n-H_n\|_p\leq \frac{1}{n}$. Then, by \eqref{desigualdadep} and \eqref{pqmenor}, we have 
	\begin{align*}
		\|Dg_n-H\|_p^p & \leq \|Dg_n-H_n\|_p^p+\|H_n-H\|_p^p  \leq \tfrac{1}{n^p}+\lambda(\Omega)^{1-p}\|H_n-H\|_1^p,
	\end{align*}
	proving that $(Dg_n)_n$ converges in $L^p(\Omega)^{d^2}$ to $H$, and then $(I_d,H)\in\mlp{\Omega}$.
\end{proof}

\begin{remark}
In this theorem, if we consider that $H$ is only Lebesgue integrable and approximate it, using simple functions approximating each component of $H$, we do not (in general) obtain functions in $SO(d)$ and so, we can not use Proposition \ref{quaseconstante}.
\end{remark}

Finally let us see that, like in the $d=1$ case, something similar to the above theorem happens if we replace the identity by another function.

\begin{corollary}
	Let $d\geq 2$, $0<p<1$, $\Omega$ a bounded  open subset of $\R^d$ and $f$ a volume preserving $C^1$ homeomorphism of $\Omega$, with $Df\in L^r(\Omega)^d$ for some $r>\frac{p}{1-p}$. Let $q=\frac{rp}{r-p}$, if $r<\infty$ and $q=p$, otherwise.
		
	Then if $(I_d,H)\in \mathcal{M}_\lambda^{1,q}{\Omega}$ then $\big(f,(H\circ f)\cdot Df\big)\in \mlp{\Omega}$.
	
	If, in addition, $Df\in SO(d)$, then $(I_d,H)\in\mlp{\Omega)}$ if and only if $(f,H\circ f)\in\mlp{\Omega)}$.
\end{corollary}
\begin{proof} Notice that, $p<q<1$, if $r<\infty$.
	Let $(f_n)_n$ be the sequence converging to $(I_d,H)$ and $g_n=f_n\circ f$. Notice that
	\begin{equation*}
		\|g_n-f\|_\infty=\max_x|f_n(f(x))-f(x)|=\max_y|f_n(y)-y|=\|f_n-I_d\|_\infty.
	\end{equation*}
	
	On the other hand, if $H=\left(H_{i,j}\right)_{i,j=1,\ldots d}$, then, using \eqref{desigualdadep},
	\begin{align*}
\|Dg_n-	(H\circ f)\,Df\|_p^p & =\max_{i,j}\left\|\sum_{k=1}^d\Big(\Big(\dfx{(f_n)_i}{x_k}-H_{i,k}\Big)\circ f\Big)\,\dfx{(f)_k}{x_j}\right\|_p^p\\
&\leq \max_{i,j}\sum_{k=1}^d\left\|\Big(\Big(\dfx{(f_n)_i}{x_k}-H_{i,k}\Big)\circ f\Big)\,\dfx{(f)_k}{x_j}
\right\|_p^p.
	\end{align*}
	
	For $i,j=1,\ldots,d$ we have, by H\"older inequality  with the conjugate exponents $\frac{q}{p}$ and $\frac{q}{q-p}$, 
		\begin{align*}
	\left\|\Big(\Big(\dfx{(f_n)_i}{x_k}-H_{i,k}\Big)\circ f\Big)\,\dfx{(f)_k}{x_j}
		\right\|_p^p&\leq 
			\left\|\Big(\dfx{(f_n)_i}{x_k}-H_{i,k}\Big)\circ f	\right\|_{q}^p\cdot	\left\|\dfx{(f)_k}{x_j}	\right\|_{\frac{pq}{q-p}}^p\\
			&=
			\left\|\dfx{(f_n)_i}{x_k}-H_{i,k}\right\|_{q}^p\cdot	\left\|\dfx{(f)_k}{x_j}
		\right\|_{r}^p,
	\end{align*}
where, in this last step, we do the change of variable $f(x)=y$, noticing that $f$  preserves the volume.
\end{proof}

\section*{Acknowledgements}
The authors  were partially financed by Portuguese Funds through FCT  - `Funda\c{c}\~ao para a Ci\^encia e a Tecnologia' within the Project
UID/00013/2025: Centro de Matemática da Universidade do Minho (CMAT/UM).

\end{document}